\documentclass{article}
\usepackage{latexsym}
\usepackage{graphics}
\usepackage{amsmath}
\usepackage{amsfonts}
\usepackage{amsthm}
\usepackage{slashed}
\usepackage{cleveref}
\begin{document}
\bibliographystyle{plain}
\title{The behaviour of the DC current density  at the edge  of electrodes}
\author{Spyros Alexakis\thanks{
alexakis@math.toronto.edu. }}
\date{} \maketitle

\def\a{{\alpha}}
\newcommand{\e}{\epsilon}
\def\b{{\beta}}
\def\be{{\beta}}
\def\ga{\gamma}
\def\Ga{\Gamma}
\def\de{\delta}
\def\De{\Delta}
\def\ep{\epsilon}
\def\eps{\epsilon}
\def\ka{\kappa}
\def\la{\lambda}
\def\La{\Lambda}
\def\si{\sigma}
\def\Si{\Sigma}
\def\om{\omega}
\def\Om{\Omega}
\def\th{\theta}
\def\ze{\zeta}
\def\ka{\kappa}
\def\nab{\partial}
\def\varep{\varepsilon}
\def\vphi{\varphi}
\def\pr{{\partial}}
\def\al{\alpha}
\def\les{\lesssim}
\def\c{\cdot}
\def\rh{{\rho}}

\def\AA{{\mathcal A}}
\def\Aa{{\mathcal A}}
\def\BB{{\mathcal B}}
\def\Bb{{\mathcal B}}
\def\CC{{\mathcal C}}
\def\MM{{\mathcal M}}
\def\NN{{\mathcal N}}
\def\II{{I}}

\def\FF{{\mathcal F}}
\def\EE{{\mathcal E}}
\def\HH{{\mathcal H}}
\def\LL{{\mathcal L}}
\def\GG{{\mathcal G}}
\def\TT{{\mathcal T}}
\def\WW{{\mathcal W}}

\def\OO{{\mathcal O}}
\def\SS{{\mathcal S}}
\def\NN{{\mathcal N}}
\def\Ss{{\mathcal S}}
\def\UU{{\mathcal U}}
\def\JJ{{\mathcal J}}
\def\KK{{\mathcal K}}
\def\Lie{{\mathcal L}}

\def\DD{{\mathcal D}}
\def\PP{{\mathcal P}}
\def\RR{{\mathcal R}}
\def\QQ{{\mathcal Q}}
\def\ZZ{{\mathcal Z}}
\def\VV{{\mathcal V}}

\def\HHb{\underline{\mathcal H}}
\def\Lie{{\mathcal L}}

\def\A{{\bf A}}
\def\B{{\bf B}}
\def\D{{\bf D}}
\def\F{{\bf F}}
\def\H{{\bf H}}
\def\I{{\bf I}}
\def\J{{\bf J}}
\def\M{{\bf M}}
\def\N{{\bf N}}
\def\L{{\bf L}}
\def\O{{\bf O}}
\def\Q{{\bf Q}}
\def\Z{{\bf Z}}
\def\R{{\bf R}}
\def\P{{\bf P}}
\def\U{{\bf U}}
\def\V{{\bf V}}
\def\S{{\bf S}}
\def\K{{\bf K}}
\def\T{{\bf T}}
\def\E{{\bf E}}
\def\X{{\bf X}}
\def\g{{\bf g}}
\def\m{{\bf m}}
\def\t{{\bf t}}
\def\u{{\bf u}}
\def\p{{\bf p}}
\def\SSS{{\mathbb S}}
\def\RRR{{\mathbb R}}
\def\MMM{{\mathbb M}}
\def\CCC{{\mathbb C}}
\def\f12{{\frac 1 2}}
\def\ub{\underline{u}}

\def\dual{{\,\,^*}}
\def\div{{\mbox div\,}}
\def\curl{{\mbox curl\,}}
\def\lot{\mbox{ l.o.t.}}
\def\mb{{\,\underline{m}}}
\def\lb{{\,\underline{l}}}
\def\Hb{\,\underline{H}}
\def\Lb{{\,\underline{L}}}
\def\Yb{\,\underline{Y}}
\def\Zb{\,\underline{Z}}
\def\Omb{\underline{\Om}}
\def\Deb{\underline{\De}\,}
\def\deb{{\,\underline{\de}\,}}
\def\Kb{{\,\underline K}}
\def\NNb{\underline{\NN}}

\def\Xh{\,^{(h)}X}
\def\Yh{\,^{(h)}Y}
\def\trch{{\mbox tr}\, \chi}
\def\trchb{{\mbox tr}\, \chib}
\def\chih{{\hat \chi}}
\def\chib{{\underline \chi}}
\def\chibh{{\underline{\chih}}}
\def\phib{{\underline{\phi}}}
\def\etab{{\underline \eta}}
\def\omb{{\underline{\om}}}
\def\bb{{\underline{\b}}}
\def\aa{{\underline{\a}}}

\def\ul{{\underline{l}}}
\def\xib{{\underline{\xi}}}
\def\th{\theta}
\def\thb{{\underline{\theta}}}
\def\va{\vartheta}
\def\vab{{\underline{\vartheta}}}
\def\Omm{\overline{m}}
\def\Mm{m}
\def\Ll{\underline{l}}
\def\Kk{l}
\def\rhod{{\,\dual\rho}}
\def\f{\widetilde{f}}
\def\W{\widetilde{W}}
\def\G{\widetilde{G}}

\def\um{\underline{u}}
\def\Lp{L}
\def\Lm{\underline{L}}
\abstract{We study the complete electrode model boundary condition for 
second order elliptic PDE. A specific case of this is the PDE describing 
the electrostatic potential for a conductive body into which current is 
injected through electrodes that touch the boundary. We obtain the optimal 
description of the gradient of the electrostatic potential upon approach to the edge of the electrodes.}

\newtheorem{proposition}{Proposition}
\newtheorem{theorem}{Theorem}
\newcommand{\Sum}{\sum}
\newtheorem{lemma}{Lemma}
\newtheorem{observation}{Observation}
\newtheorem{formulation}{Formulation}
\newtheorem{definition}{Definition}
\newtheorem{conjecture}{Conjecture}
\newtheorem{corollary}{Corollary}
\section*{Introduction}
 We consider the problem of understanding the asymptotics of 
solutions to second order elliptic PDE subject to the \emph{Complete Electrode Model} boundary conditions. 

The PDEs we consider are over a simply connected domain 
  $\Omega\subset \mathbb{R}^n$, where $\Omega$ has a 
  ${\cal C}^{\lceil \frac{n+2}{2}\rceil}$ boundary, and are defined as solutions to 
\begin{equation}\label{the PDE}
{\cal L} u(x)=\rho(x)
\end{equation}
 where $\cal L$ is a strictly elliptic second order partial differential operator with ${\cal C}^{\lceil \frac{n+2}{2}\rceil }$ coefficients; the principal symbol will be  a
     Laplace-Beltrami operator associated to a ${\cal C}^{\lceil \frac{n+2}{2}\rceil }$ Riemannian metric $g$.

     So $\cal L$ will be of the form: 
     
     \begin{equation}\label{operator}
    {\cal  L}=\frac{1}{\sqrt{|g|}}\sum_{i,j=1}^n \partial_i 
    (\sqrt{|g|} g^{ij}(x)\cdot \partial_j)+\sum_{i=1}^n a^i(x)\cdot
     \partial_i +V(x)
     \end{equation}
Here, letting $d= \lceil \frac{n+2}{2}\rceil $,  $g^{ij}(x)$ (for all $x\in\Omega$)  is a symmetric $n\times n$ matrix which is positive 
definite and  with entries that are 
${\cal C}^{d+1 }$ smooth,  $a^i(x)$ is a 
${\cal C}^{{{d}}}$  vector field and $V(x)$ is a ${\cal C}^{{{d}}}$ scalar.  $\rho(x)$ 
is a ${\cal C}^{{{d}}}$ source term. $g_{ij}(x)$ (the inverse of $g^{ij}(x)$) 
defines  a Riemannian metric and $|g|$ is short-hand notation for 
${\rm det}(g_{ij}(x))$ 

The most classical example of such operators would be the standard 
Laplacian, where $g_{ij}=\delta_{ij}$ and $a^i(x)=V(x)=0$. 
In that setting, $u(x)$ would stand for the electrostatic potential 
induced from the forcing term 
$\rho(x)$ (which models a  charge density  inside $\Omega$),  as well as 
the current that is induced on the boundary, which we review momentarily. 
The description of $u$ as an electrostatic potential is also valid 
 when $a_i=V=0$; in that case $u$ is the potential induced on a material 
 of (an-isotropic) conductivity 
 $\sigma^{ij}(x)= \sqrt{|g|}\cdot g^{ij}(x)$. The vector 
 \[
 \sigma^{ij}(x)\cdot \partial_j u
 \]
 stands for the \emph{current density}. The main object of this paper is 
 to understand this current density when the boundary condition is given 
 by the \emph{complete electrode model}: In dimensions $n=2,3$
this captures 
the real-world setting
of a finite number of electrodes touching the conductive body $\Omega$. 
Since our argument carries over verbatim to the presence of first and 
zeroth order terms in the operator $\cal L$,  \eqref{operator}, we present the proof 
in this more general setting.
  
The electrodes correspond to a finite number of relatively open subsets  
$E_i\subset\partial\Omega$, 
        with smooth boundaries (in $\partial \Omega$).  Let $\nu$ be the 
        unit 
        normal vector field to $\partial\Omega$ (pointing into
         $\Omega$). \footnote{``Unit'' will be with respect to the 
         Riemannian metric $g$ (in 
        the special case where ${\cal L}=\Delta$, $\nu$ will be the unit 
        normal with respect to the Euclidean metric).  }

The CEM boundary condition  is then: 

\begin{equation}\label{CEM bdry}
\int_{E_i}\partial_\nu udV_{g|_{\partial M}}=J_i , u|_{E_i}={U}_i, 
\text{ }\partial_\nu u=0\text{ } {\rm on}\text{ } \partial \Omega\setminus 
(\bigcup_{i=1}^N E_i).
\end{equation}
Here the numbers $J_i$ are \emph{prescribed}, subject to the constraint 
$\sum_{i=1}^NJ_i=0$.   They encode the pattern of \emph{DC current} that we inject 
through the electrodes.   The constants
 $U_i$ are \emph{not} 
prescribed--they represent the restriction that the voltage on each of the 
electrodes  should be constant, 
but one does not get to prescribe the constant. 
The solution $u$ is only defined up to an additive 
constant.  To restore uniqueness,  one can 
 impose a ``grounding'' normalization,  e.g. $\int_\Omega udV_g=0$,  although this is not 
 necessary for what we do here. 

We remark that one often considers a variant of the conditions above where the condition  $u|_{E_i}={U}_i$ is replaced by a condition 
$u+z_i\cdot \sigma \partial_\nu u|_{E_i}=U_i$, where $U_i$ is again an (un-prescribed) constant; the coefficients $z_i$ 
are fixed by the problem and denote the \emph{contact} resistances. 
In our work here we treat the case of all
 contact resistances being zero. \footnote{A variant of the boundary conditions would be that 
we \emph{define} the constants ${U}_i$ and then the numbers $J_i$ are forced--this should 
be seen as maintaining a constant potential on the electrodes and then inducing DC
current through the electrodes as a result.}

The existence and uniqueness for \eqref{the PDE} subject to \eqref{CEM bdry}
when $a^i(x)=V(x)=\rho(x)=0$ 
has been established for $u\in H^1(\Omega)$ in \cite{SomChenIsaac}.
That argument generalizes readily in the presence of a right hand side  
$\rho\in L^2(\Omega)$.   For general lower-order terms $a^i(x)$ and $V(x)$ existence 
and uniqueness may fail, so they need to be imposed as an apriori 
assumption.  Our goal in this paper is to understand the regularity of $u$ (and $\nabla u$) on $\overline{\Omega}$:

The smoothness of the solution $u$  in the interior ${\rm Int}(\Omega)$ 
is very classical. The (local)  boundary regularity of $u$ upon approach 
to either points in the interior of an electrode, ${\rm Int}(E_i)$ or a 
point on the exterior of all electrodes,
$\partial\Omega\setminus(\bigcup_{i=1}^N\overline{E_i})$, is also 
classical,  since one 
has either constant Dirichlet or zero Neumann 
conditions around any such point--in the particular the solution extends 
in a 
${\cal C}^d$-fashion to those portions of the boundary, given the regularity assumptions on 
$\partial\Omega, g, a^i, V, \rho$ here.    

Our goal in this paper is to  to  understand the asymptotic behaviour  of 
$\nabla u$ upon approach to the \emph{edge} 
$\partial E_i$ of each electrode.  This is motivated by the real-world 
appearances of this 
setting.  One is in electrical impedance tomography, see for example
\cite{Darbasetal, HAKULA, Hyvonen, Vauhkonenetal}.  
But even for fully
 passive measurements (where $J_i=0$, $\forall i=1,\dots, N$), 
such as those  performed in an electro-cardiogram,  one would like to know 
the behaviour
 of the  current density at the edge of each electrode.  Another 
 motivation to 
 pursue this question is the desire to develop numerical solvers for the 
 equation \eqref{the PDE} subject to \eqref{CEM bdry}.
In fact the difficulty of obtaining accuracy 
  for both finite element and boundary element  
  numerical solvers due to the singularity of the solution precisely at the edge of the electrodes
   has been noted in \cite{DelKress}.  We expect that 
the apriori knowledge of the asymptotic behaviour of the solution near 
 the edges $\partial E_i$ can be used to increase the accuracy 
 of  numerical solvers.   
 \medskip

 Beyond EIT and the CEM boundary condition, we should mention that the problem  considered here can be seen 
 as a Zaremba problem,  with mixed Dirichlet and Neumann boundary conditions.  It thus fits into a 
 classical circle of questions going back to \cite{Zaremba}.  
 Such mixed boundary value problems 
 describe phenomena beyond CEM electrostatics; in particular  in elasticity of partially fastened bodies
\cite{AlChMa, Eskin} 
  and more (\cite{AlChMa}).  A prototypical mixed-boundary value problem seeks to solve an elliptic PDE inside a domain 
  where the boundary condition on parts of the domain is Dirichlet and on the remaining part of 
  the boundary  it is Neumann.  (Note that for CEM
  it  is vanishing Neumann and piecewise constant Dirichlet).

 To put this in context,  let us discuss some prior result on boundary regularity in mixed-boundary problems 
 that are similar to ours: In two dimensions, and only for constant coefficient elliptic operators with no
 lower-order terms (say the Laplacian for simplicity) we mention two results which are related 
 to the simpler first result we derive: 
A result which is almost identical to our simpler problem in two dimensions is obtained in section 13 in \cite{Eskin}, 
combining example 13.1 with part 6 of that section --however the proof is not elementary as the 2-dimensional proof here.  
Still in 2 dimensions,  
Grisvard \cite{Grisvard} studied harmonic functions $u$ over polygons (in $\mathbb{R}^2$) 
where at given corners one is allowed to have an alteration of 
vanishing Dirichlet and vanishing Neumann conditions.
He showed the existence of a power $p$ so that (roughly speaking) $|\nabla u |\sim r^{-p}$ upon approach to such 
a juncture between vanishing Neumann and vanishing Dirichlet boundary conditions; however the power $p$ is not found. 

On more general variable-coefficient PDEs the results in the literature do not provide the strong control we
 obtain at the juncture between the Neumann and Dirichlet condition that we obtain here. 
A theory which is also applicable to variable coefficients (and  higher dimensions) is developed in 
\cite{RempSchulzePaper}, as well as section 24 in \cite{Eskin}.  The regularity  one obtains is not strong enough
to obtain leading order asymptotics at points on the juncture between the two types of boundary conditions. 

 While the machinery in those papers 
is taylored for more general questions  than what is developed here, 
if one were to invoke those results to the case at hand, the precise power 
of blow-up of $|\nabla u|$ would not be determined by those methods. 
  The type of 
analysis in \cite{Eskin, Grisvard, RempSchulzePaper} is most frequently applied 
in the setting of boundary value problems on manifolds with singularities
 (such as corners, point singularities, more general singularities).   
 A few references to such problems are the pioneering \cite{Kondratev} 
 and 
 \cite{KozlovMazyaRossmann}. 
We note that mixed boundary problems for elliptic equations  are also of interest for much less regular regions where the Neumann and Dirichlet data 
live; for instance for zero Neumann and Dirichlet data one can obtain estimates for $|\nabla u|$ in $L^{2+\delta}$ spaces--see the recent paper 
\cite{AlChMa} and references therein. 
 \medskip

We will first state and prove our result in the simpler setting of the Poisson equation in 
dimension $2$.  In that case the proof is more concise and elegant, and entirely reduces to complex analysis. 
After that we will state and prove the claim for general dimensions and the more general operators  $\cal L$
as  in \eqref{operator}.

\subsection*{Edge of the electrode asymptotics for the Poisson equation in two dimensions}  
To state our result for $\Omega\subset \mathbb{R}^2$ with ${\cal L}=\Delta$,   consider  an arc-length parameter $s$ of  the boundary $\partial\Omega$ 
  which allows us to think of it as an interval
  $[0,L]$  of length $L$ (with the endpoints identified). 
  
  The CEM boundary condition considers $N$ ``electrodes'' through which DC 
  current is injected onto the boundary. 
  The electrodes correspond to $N$  disjoint closed sub-intervals on the 
  boundary $\partial\Omega$, which (only for this section) we denote by  
  $e_1,\dots, e_N\subset \partial\Omega$  
  
   Recall $\nu$ be the unit normal vector field to $\partial\Omega$ (with 
   respect to the Euclidean metric now).  
   Then the CEM boundary condition is that on each 
   $e_i$: 
   
   \[
   \int_{e_i}\partial_\nu uds=J_i,  u(s)=V_i,
   \]
     where $J_i, V_i\in\mathbb{R}$. In addition  $\partial_\nu u(s)=0$ 
     for all $s\in \partial\Omega\setminus \bigcup_{i=1}^Ne_i$.  As 
     already explained, 
the \emph{choice} of data are the $N$ numbers $J_i$ subject to     the 
restriction $\sum_{i=1}^N J_i=0$. 
The constants $V_i$ are \emph{not} prescribed  as free data; they form 
part of the condition that the electrostatic 
potential should be constant  on each electrode.
To ensure uniqueness, we can impose $\sum_{i=1}^NV_i=0$.

 Our aim here is to derive 
the optimal asymptotic behaviour for $u$ and $\partial_\nu u$ at the 
endpoints of  the electrodes $e_i$. 

\begin{theorem}\label{Them 2D}
Consider the PDE \eqref{the PDE}  when $n=2$ and ${\cal L}=\Delta$ (the 
Euclidean Laplacian), and $\rho(x)$ is 
compactly supported away from $\partial\Omega$. 

Consider any electrode $e_i$ which corresponds to an interval 
$[s_i, s_{i+1}]\subset [0,L]$.  
Then consider the function $\partial_\nu u$ on $(s_i,  s_{i+1})$. 
This function 
satisfies the following asymptotic expansion: 
As $s\to s_i^+$ and $s\to s_{i+1}^-$ (i.~e.~approaching $\partial e_i$ 
from ${\rm int}(e_i)$):
\[
\partial_\nu u(s)\sim |s-s_i|^{-1/2}, \partial_\nu u (s)
\sim |s-s_{i+1}|^{-1/2}
\]
while for points $s\notin e_i$ as $s\to s_i^-$ or $s\to s^+_{i+1}$
 (i.~e.~approaching $\partial e_i$ from ${\rm Ext}(E_i)$):
\[
\partial_s u(s)\sim |s-s_i|^{-1/2},  \partial_s u(s)\sim |s-s_{i+1}|^{-1/2}.
\]
\end{theorem}

\begin{proof}
 We prove our result for the left  endpoint $A$ of $e_i$. The proof for the 
 right endpoint
follows by an obvious modification. 
We first invoke the Riemann mapping theorem,  to 
map $\Omega$ to the upper-half 
 plane $\mathbb{R}^2_+$, sending a point 
 $P\in \partial\Omega\setminus\bigcup_{i=1}^Ne_i$ to infinity.
  (Denote the 
 Riemann mapping by $\cal R$ for future reference). 
We recall (see \cite{Pommerenke} for instance) that the Riemann
 map extends to $\partial\Omega$ in a ${\cal C}^1$ way so 
$M^{-1}\le |{\cal R}^\prime(z)|\le M$ in a neighborhood of $A$  for some 
$M>1$. 
 
   The 
 electrodes 
are then mapped to closed subintervals 
$\tilde{e}_i\subset \mathbb{R}=\partial\mathbb{R}^2_+$. By abuse of 
notation, 
we denote them by $e_i$ again.  By translation, we can assume that 
 $e_i$ is of the form 
$[0, l_i]\subset\mathbb{R}$.  The function $u$ composed with the Riemann mapping, $u\circ {\cal R}$
 is again harmonic in 
$\mathbb{R}_+^2$ and we denote it by $u$ again, by abuse of notation. 
Using the coordinates $(x,y)$ in $\mathbb{R}^2_+$ we see that $\partial_x u(x,0)=0$ for $x\in (0, l_i)$ 
and $\partial_y u(x,0)=0$ for $x<0$ and small enough in magnitude.  We note also that $\nabla u\in L^2(D_+(0, l_i/2))$, where 
$D_+(0, l_i/2)$ is the half-disc centered at the origin of radius $l_i/2$.  This follows from the invariance of 
the $\dot{H}^1$ norm under the Riemann mapping.

The proof will use complex analysis. Consider the complex-valued function 

\[
f(x,y)=\partial_x u+i\cdot \partial_y u. 
\]
As is well-known, this function is holomorphic for $y>0$, 
since the harmonicity of $u$ implies $f(x,y)$  satisfies the Cauchy-Riemann equations.  We thus write $f(z), z=x+iy$.  Note that the 
${\cal C}^1$ 
smoothness of $u$ away from the endpoint of the electrode  implies that 
$f$ extends continuously to the $x$-axis and the $y$-axis,  \emph{away} from the origin. 
We note that 

\begin{equation}
\label{H1 to L2}
\|f\|_{L^2(D_+(0,l_i/2))}= \|u\|_{\dot{H}^1(D_+(0,l_i/2))}<\infty. 
\end{equation}

Choose $\delta$ small enough so that $[0,\delta]\subset e_1$ and 
$[-\delta, 0)\bigcap (\bigcup_{i=1}^N{e_i})=\emptyset$. 

We wish to study the behaviour of $u$ in $D_+(0,\delta)$ as one approaches the origin. 
Let us use polar coordinates $(r,\theta)$ on $D_+(0,\delta)$, so
 $z=r\cdot e^{i\theta}, \theta\in [0,\pi]$, $r\in (0,\delta)$.

We further note that

\[Re(f)(z)=\partial_x u=0, \text{ } \forall z= x+i\cdot 0, x\in (0, \delta), \text{ }{\rm and}\text{ } Im(f)(z)=\partial_y u=0, \text{ } \forall z=x+i\cdot 0,  x\in (-\delta,0)
\] 
\medskip

Now, define $\zeta=\sqrt{z}$, via $\zeta=\sqrt{r}e^{i\cdot\theta/2}$.  Then we obtain a holomorphic function $F(\zeta)$ 
on the set 

\[
D_{++}(\sqrt{\delta}):=\{ x^2+y^2\le\delta, x,y> 0\}, 
\]
defined by 

\begin{equation}
F(\zeta)=f(\zeta^2). 
\end{equation}
Moreover $F(\zeta)$ extends continuously to the positive $x$-axis and the positive $y$-axis. 
Now $Re[F]=0$ on the positive $x$-axis, while $Im[F]=0$ on the 
positive $y$-axis (in $D_{++}(\sqrt{\delta})$). 

\newcommand{\z}{\zeta}
We note  the volume forms are related as follows: 
\begin{equation}
\label{vol forms}
dzd\overline{z}=4|\z|^2d\z d\overline{\z}
\end{equation}
Thus we note:

\begin{equation}
\label{L^2 norms}
 \int_{D_{++}(\sqrt{\delta})} |F(\z)|^2 |\z|^2 d\z d\overline{\z}=\frac{1}{4}\int_{D_+(0,\delta)} |f(z)|^2dzd\overline{z}<+\infty.
\end{equation}

 We will next show that $F(\z)$ admits an extension to a holomorphic 
 function $\tilde{F}(\zeta)$ 
 on the punctured disc $D(0,\sqrt{\delta})\setminus \{0\}$, with finite 
 $L^2(|\zeta|^2d\zeta d\overline{\zeta})$-norm.  (The $L^2$ norm defined with respect to the volume element $|\zeta|^2d\zeta d\overline{\zeta}$). 
 This will imply that $\tilde{F}(\zeta)$ has a simple pole of rank 1 at 
 the origin,  and  will translate to the desired claim on $f$. 
  
   We construct the holomorphic extension of $F$ explicitly: 
   We obtain this by two Schwarz reflections.  We first extend
   $F$ to the quadrant
   
   \[
   D_{+-}(0,\sqrt{\delta})= \{ x^2 +y^2\le \delta, x> 0, y\le 0\},
   \]
   by using Schwarz reflection to define 
   $\tilde{F}(\overline{\z})= \overline{F(\z)}$ on $D_{+-}(0,\sqrt{\delta})$. 
($\tilde{F}=F$ on $D_{++}   (0,\sqrt{\delta})$). 
   Using the continuity of $F(\z)$ away from the origin 
   we obtain $\tilde{F}$ holomorphic on 
   $\bigg( D(0,\sqrt{\delta})\setminus \{0\}\bigg)\bigcap \{Re(\z)> 0\}$. 
   Note that $Im[\tilde{F}(\z)]=0$  on 
    the imaginary axis, away from the origin, in view of the zero Neumann 
    property of the original function $u(x,y)$.   
   We  then perform another Schwarz reflection across the imaginary 
   axis. 
   
 First consider the new holomorphic function  
   $G(i\z):=i\cdot \tilde{F}(\z)$,  which is defined for
    $\bigg( \z\in D(0,\sqrt{\delta})\setminus \{0\}\bigg)\bigcap \{Im(\z)< 0\}$. This again extends continuously to the real axis away from $0$. 
   
Then  we  define the holomorphic extension  $G(\z) $ on $D(0,\sqrt{\delta})\setminus \{0\}$ via: 
       $G(\overline{\z})=\overline{G(\z)}$. 
       Finally we  define 
       
       \[
       \tilde{F}(-i\cdot \z)= -iG(\z)\text{ } {\rm for} \text{ }\z\in \bigg( D(0,\sqrt{\delta})\setminus \{0\}\bigg) \bigcap \{Re \z\le 0\}.
       \] 
       
        The resulting function $\tilde{F}(\z)$  is holomorphic in the punctured disc $D(0,\sqrt{\delta})\setminus \{0\}$. 
       
We next seek to understand the pole of $\tilde{F}$ at the origin. 
By construction of the two extensions: 
        
\[
\|\tilde{F}\cdot|\z|\|_{L^2\big[D(0,\sqrt{\delta})\big]}=4 \|\tilde{F}\cdot|\z|\|_{L^2\big[D_{++}(0,\sqrt{\delta})\big]}. 
\] 
        Now, recalling \eqref{L^2 norms}, \eqref{H1 to L2} we see that 
                                                        $\|\tilde{F}\cdot |\z|\|_{L^2\big[D(0,\sqrt{\delta})}<\infty$. 
        It follows  (by the Riemann removable singularity theorem) 
        that $\tilde{F}(\z)\cdot \z$ is holomorphic. 
        
         In other words,  there exists an $A\in\mathbb{C}$ so that: 
         
         \[
        \tilde{F}(\z)=\frac{A}{\z}+r(\z), 
\]        
where $r(\z)$ is a holomorphic (and thus smooth) function on $D(0,\sqrt{\delta})$.

       Translated to $f(z)$ (and ultimately $u(x,y)$), 
       this means: 
       
       \begin{equation}\begin{split}
&      \partial_x u(x,y)+i\cdot \partial_y u(x,y)= f(z)= \frac{A\cdot \sqrt{\overline{z}}}{\|z\|}+r(\sqrt{z})
\\&=A\cdot r^{-1/2}   \cdot \left(\cos(\frac{\theta}{2})-i\cdot \sin(\frac{\theta}{2})\right)+r(\sqrt{z})
\end{split}   
       \end{equation}
       Using the fact that $\partial_y u(x,0)=0$ for $x<0$ we derive that $A$ must be purely imaginary.  We
        write $A=i\cdot\alpha$, $\alpha\in\mathbb{R}$. 
       
       We let $R(x)$ be a generic ${\cal C}^{1/2}$ function over $[0,\delta)$ or $(-\delta, 0]$ respectively. 
       
Approaching from \emph{inside} the electrode, i.e. for $x>0$: 
        
        \[
        \partial_y u(x, 0)=\alpha\cdot x^{-1/2}+ R(x), 
        \partial_x u(x,0)= 0.
        \]
      Approaching $x=0$ from the \emph{left} (i.e.~from \emph{outside} the electrode)  we derive:

      \[
        \partial_y u(x,0)=0,  \partial_x u(x,0)= \alpha |x|^{-1/2}+R(x). 
        \]  
        
        Our claim then follows, by composing $u$ with the inverse of the Riemann mapping $\cal R$, and noting that $|\frac{dx}{ds}|$ on $\partial \Omega$ is uniformly bounded above and below. 
        The latter follows since ${\cal R}^\prime(z)$ and $({\cal R}^\prime(z))^{-1}$  are
         uniformly bounded and uniformly bounded away from zero. 
        \end{proof}

        \subsection*{Edge of the electrode  asymptotics in the
         general setting.}
        
 Here we state and prove the result in general dimensions $n\ge 2$ and for 
the more general PDE \eqref{the PDE} with $\cal L$ being in the 
form 
\eqref{operator}. 

Since our result concerns the local behaviour of $u$ near any point $P$ on the edge of an electrode $E_i$ 
we can have $\cal L$ being defined over any 
  compact manifold $M$ with smooth boundary $\partial M$, and then $g$ would be  
a Riemannian metric on this manifold-with-boundary. So we use $M$
 instead of $\Omega$ in this section.  $\nabla u$ will be the gradient of $u$ with respect to the metric $g$. 
 Also the function spaces $L^2(M),  H^1(M)$ below are with respect to the metric $g$, unless stated otherwise. 
 (Further down in the proof this will no longer be the case, as will be clear by our notation).

We impose the CEM boundary condition with respect to a finite number of electrodes $E_i$ as in \eqref{CEM bdry}. 
We will be making the assumption that the problem \eqref{the PDE} subject to \eqref{CEM bdry} has a unique solution in $H^1(M)$ 
for all $\rho\in L^2(M)$. 
(This is true if $a^i(x), V(x)=0$ in view of \cite{SomChenIsaac}--in general it is an additional condition on the operator $\cal L$ that needs to be imposed).  

Our goal is to derive a more precise understanding of $\nabla u$ upon approach to the edge $\partial E$ of any electrode $E$. 
In higher dimensions,  our result can be stated after we introduce  a suitable coordinate system.

We commence with choosing coordinates $x_1, \dots x_{n-2}$ \emph{on} $\partial E$ near $P$. 
To partly normalize,  we require $x_1(P)=\dots =x_{n-2}(P)=0$. 
 We extend these by adding an $(n-1)$-coordinate $x_{n-1}$ on $\partial M$ so that $\{x_{n-1}=0\}\subset \partial E_i$. 
 Finally we add a final coordinate $x_n$ on $M$ so that $\{x_n=0\}\subset M$, and $x_n>0$ on ${\rm int}(M)$. 
 The only further normalization we impose is that at $P$ the  metric $g_{ij}$ (ith respect to these coordinates) 
 takes the form $g_{ij}(P)=\delta_{ij}$.

Our Theorem in this general case is the following: 

\begin{theorem}
\label{gen thm}
Consider a solution to \eqref{the PDE} under the regularity assumptions in the introduction. 
Assume that the operator $\cal L$ in \eqref{operator} has the property that for any $\rho\in L^2(M)$ there exists a solution $u\in H^1(M)$ 
to \eqref{the PDE}, for any CEM boundary condition data.  

Then, for any point $P\in \partial E$ for any electrode $E$ it follows that the gradient $\nabla u$ satisfies the following 
bound upon approach to the point $P$ (in the coordinates just constructed): 

 There exits a value $A\in \mathbb{R}_+$ so that as $x_{n-1}\to 0$:
\[
 |\nabla u|_g (x_1,\dots, x_{n-2},x_{n-1}, 0)=A|x_{n-1}|^{-1/2}+O(1).
 \]
 In fact more generally, this holds upon approach to $N$ even from inside the manifold $M$: As $(x_{n-1}, x_n)\to 0$
 \begin{equation}\label{asympt claim}
 |\nabla u|_g (x_1,\dots, x_{n-2},x_{n-1}, x_n)
 =A\bigg[\sqrt{(x_{n-1})^{2}+(x_n)^2}\bigg]^{-1/2}+O(1).
 \end{equation}
 
\end{theorem}   
\newtheorem{remark}{Remark}
   \begin{remark}
   In a suitable coordinate system, the claim can be made more precise--see \eqref{stronger claim} below. 
   \end{remark}

   \begin{proof}
   Our proof proceeds in steps. 
   
   Initially we will choose any point $Q$ on the boundary of the electrode and construct  a preferred coordinate system around $Q$. 
   The key feature of this will be that the metric $g$ has a splitting between $n-2$ dimensions that ``move'' along the boundary of
    the electrode and the remaining two dimensions which are transverse (and in fact normal,  with respect to $g$)  to the boundary of the electrode. 
    
In  the second step,
given the higher assumed regularity of $g_{ij}(x), a^i(x), V(x),\rho(x)$ we then 
 note that the regularity of $u$ one initially has (in $H^1(M)$) can be upgraded to a suitable Sobolev space 
which ``sees''  the $n-2$ derivatives in the directions $x_1,\dots, x_{n-2}$ up to order ${{d}}$. 
    (We measure these in $H^1(M)$, and we call them the ``tangential directions'', since they are tangential to $\partial E$). 
After this has been accomplished, we can 
treat the equation \eqref{the PDE} as a PDE in the \emph{two} 
dimensions $(x_{n-1}, x_n)$, where all other derivatives of the unknown 
$u$ have already been bounded, and are moved to the RHS (making use of the trace theorem).  
The analysis in the two dimension mimics the complex-analytic techniques we used in the previous section.

 At the third step we obtain our asymptotic expansion.  The argument here makes use to the technique 
 in complex analysis we did in dimension two,   but is more analytic: We introduce isothermal coordinates 
 to replace the $x_{n-1}, x_n$ coordinates, yielding  a complex coordinate $z$.  
 Then,  we perform a change of the differential structure, 
 based  on the transformation $\zeta=\sqrt{z}$.  This allows us 
 to perform even and odd reflections to obtain an elliptic PDE on a punctured disc. 
 The desired estimates are obtained by use of Green's functions, obtaining bounds in $W^{2,p}$ spaces, with $p>2$. 
 Our result in the end is obtained by these estimates and Sobolev 
 embedding. 
 \medskip
\medskip

 The  normalization of coordinates comes first: We consider our chosen
  point $P\in\partial E$. We initially construct $(n-2)$ coordinates 
 on $\partial E$ around $P$,  which we denote by $(x_1,\dots x_{n-2})$,
 with $P$ having coordinates $(0,0,\dots, 0)$.

 We extend these by adding an $(n-1)$-coordinate $x_{n-1}$ on 
 $\partial M$ so that $\{x_{n-1}=0\}\subset \partial E$  and $\partial_{n-1}$ points into $E$. 
 Finally we add a final coordinate $x_n$ on $M$ so that 
 $\{x_n=0\}\subset M$, and $x_n>0$ on ${\rm int}(M)$. In particular the
  point $P$ has  coordinates $(0,0,\dots ,0)$. 
 
We impose a slight further normalization by requiring that  on $\partial E$ near $P$,

\begin{equation}\label{normality}
\forall j\in \{1,\dots n-2\}\text{ }g_{j(n-1)}=g_{jn}=0\text{ }\&\text{ } g_{AB}=\delta_{AB},\text{ }{\rm for} 
\text{ }
A,B\in \{n-1,n\}. 
\end{equation}
  Let $\Sigma(P)=\{x_{n-1}=0, x_n=0\}$ be the 2-surface which is then normal to $\partial E$ at $P$.

{Moreover},  we normalize the coordinates $x_1,\dots, x_{n-2}$ off of $\Sigma(P)$ 
 to ensure that on $\Sigma(P)$ the cross terms $g_{Ai}=0$ for all $A=n-1, n$ and $i=1,\dots, n-2$.  This can be achieved by altering the 1-jet of the 
 coordinates $x_1,\dots, x_{n-2}$ off of $\Sigma(P)$, by adding suitable multiples of $\partial_{n-1}, \partial_n$ at each point.

  The metric $g$ can then be expressed in these coordinates as:
  
   \[
   g=\sum_{i,j=1}^{n-2} g_{ij} dx^i dx^j +\sum_{A,B=n-1,n} g_{AB} dx^Adx^B
   +\sum_{i=1}^{n-2} \sum_{A=n-1}^n g_{Ai} dx^Adx^i.
   \]
Where the cross terms $g_{Ai}=0$ on $\Sigma(P)$ (i.e. on $\{x_{n-1}, x_n=0\}$). 
   \medskip
   
   We will then show the following statement, which implies 
   (and in fact strengthens) \eqref{asympt claim}:
   
 In this coordinate system, for all indices  
 $i\in \{1,\dots n-2\}$  $\partial_i u$ is a ${\cal C}^\alpha$ 
 function. The other components behave as follows: 
 
  Let $R$, $\varphi$ be defined by $R=\sqrt{(x_{n-1})^2+(x_n)^2}$,   
  $\varphi={\rm arcsin}\frac{x_n}{\sqrt{(x_{n-1})^2+(x_n)^2}}$.
Then there exists continuous functions $A(x_1,\dots x_{n-2})$,  $b(x_1,\dots, x_n)$ so that: 

 \begin{equation}\begin{split}\label{stronger claim}
&\bigg( \partial_{n-1} u,  \partial_n u\bigg) (x_1,\dots x_n)= 
A(x_1,\dots x_{n-2})  \bigg(R^{-1/2} \sin(\varphi/2), 
R^{-1/2} \cos(\varphi/2)\bigg) 
\\&+b(x_1,\dots, x_n).
\end{split} \end{equation}

   The fact that \eqref{stronger claim} implies \eqref{asympt claim}
   follows by simple Riemannian geometry, by just checking that for any point $Q$ in our coordinate neighborhood there exits a fixed $B$ so that: 
   
   \[
 B^{-1} \le R^{-1/2}\cdot [{\rm dist}(Q,P)]\le B. 
   \]
   This is immediate from our construction of our  coordinates.
   \medskip

We now move to the second step. 
We start by constructing a cylindrical  neighborhood of $P$ in the coordinates $x_1, \dots, x_n$ just constructed: 
Given the  point $P\in \partial E$ on the boundary of the electrode $E$, 
we then let $D^2_P(0,2\epsilon)$ stand for the half-disc in the 
$(x_{n-1}, x_n)$ half-plane centered 
at 
the origin, with radius $2\epsilon$.  We then also denote the solid cylinder 
$D^2_P(0,2\epsilon)\times [-2\epsilon,2\epsilon]\times [-2\epsilon,2\epsilon]$ by ${\rm Cyl}(2\epsilon)$.  For 
$\epsilon>0$ 
small enough the solid cylinder ${\rm Cyl}(2\epsilon)$ is contained in our coordinate patch around $P$.  
We will be making this choice of $\epsilon$  from now onwards. 

Consider the derivatives $\partial^C_{ij} u$,  $C=1,2$,  $ i,j\in \{1,\dots, n-2\}$.
We claim that these derivatives upon restriction to any disc $D^2_P(0,\epsilon)$
 lie in $H^1(dx_{n-1}dx_n)$:

\begin{lemma}\label{restriction}
The restriction of any derivative $\partial^C_{ij} u$,  $C=1,2$,  $ i,j\in \{1,\dots, n-2\}$  to 
the disc $D^2_P(0,\epsilon)$ lies in 
$H^1(D^2_P(0,\epsilon))$ (with respect to the volume form $dx_{n-1}dx_n$).  
\end{lemma}
Once this Lemma has been obtained
we will be able to treat \eqref{Lapla in 2} as a Poisson equation (over a half-disc) and understand the asymptotics of the solution 
at the origin.   
  \medskip
  
  \begin{proof}
  
We obtain our $H^1$ estimates for $\partial^C_{ij} u$, $C\ge 2, i,j\in \{1,\dots, n-2\}$
on $D^2(0,\e)$  by obtaining estimates for higher 
 derivatives of 
$u$ (in the directions $\partial_k, k\in \{1,\dots, n-2\}$)
over ${\rm Cyl}(\epsilon)$, in conjunction with the trace theorem.

  The argument goes as follows: Consider a smooth cutoff function $\chi(q)$ with $\chi(q)=1$ for $|q|\le \epsilon$ 
  and $\chi(q)=0$ for $|q|\ge 3/2\epsilon$.  Consider the vector fields 
  $X_i=\chi (x_i)\cdot \partial_{x_i}$, $i=1,\dots, n-2$.  
  (Unless stated otherwise, $i\in \{1,\dots, n-2\}$ in any $X_i$ from now on). 
  We claim that for any multi-index $I= (i_1, \dots, i_k)$ with $i_j\in \{1,\dots, n-2\}$ and 
  $k\le {{d}}$ 
    we have 
  $\partial^Iu\in H^1({\rm Cyl}(2\epsilon))$.  We derive this iteratively: 
  
  First take one $X_i$ derivative of $u$. We derive an equation: 
  
   \begin{equation}\label{Diffd PDE}
  - \Delta_g [X_i u]= \sum m(\partial^2 u, \partial u)
   \end{equation}
   Here $m(\partial^2 u,\partial u)$ stands for  a generic linear combination of terms 
   $m^{\alpha\beta}(x)\cdot \partial_{\alpha\beta}^2u, m^\alpha(x)\cdot \partial_\alpha u$,        
where $m^{\alpha\beta}, m^\alpha$ are generic ${\cal C}^{{{d}}-1}$ functions with support in ${\rm Cyl}(2\epsilon)$. 

Let us  note that by the CEM boundary condition on $u$, we know that $X_iu=0$ on $E_i$ and 
$\partial_\nu [X_i u]=0$ on $\partial\Omega\setminus (\bigcup_{i=1}^NE_i)$. 
(Note that by construction $\nu=\partial_{x_n}$ on $\partial M$ near $P$, thus $[\nu, X_i ]=0$). 

 We then multiply 
        \eqref{Diffd PDE} by $X_i u$ and integrate the resulting equation over ${\rm Cyl}(2\epsilon)$.  
        We then sum the resulting equations for each $i=1,\dots, n-2$. 
We denote the equation we obtain schematically by: 

\begin{equation}
\sum_{i=1}^{n-2} \int_{{\rm Cyl}(2\epsilon)}{\rm LHS} \bigg[\eqref{Diffd PDE}  \bigg] \cdot (X_i u) dV_g= 
\sum_{i=1}^{n-2} 
 \int_{{\rm Cyl}(2\epsilon)}{\rm RHS} \bigg[\eqref{Diffd PDE}  \bigg] \cdot (X_i u) dV_g.\label{expanded}
\end{equation}        
        
        We perform integration by parts in the LHS  to derive:\footnote{We can introduce a mollification of $X_i u$ to justify this step strictly--we skip the details.  }
       
       \begin{equation}\label{IBP}
      - \int_{{\rm Cyl}(2\epsilon)}X_iu\cdot \Delta_g [X_i u]dV_g =\int_{{\rm Cyl}(2\epsilon)}|\nabla [X_iu]|_g^2 dV_g.
       \end{equation}
         The boundary terms away from $\partial M$ vanish in view of the vanishing of $X_i$ there. 
         The boundary terms on $\partial M$ vanish in view of the boundary conditions of $X_i u$ 
         on $E_i$ and off of $E_i$ (vanishing Dirichlet and Neumann data, respectively).  
   
   On the other hand, in the RHS of the equation we claim that for some $D$ large enough (depending on the 
   ${\cal C}^1$-norm of $m^{\alpha\beta}, m^\alpha$)
   
   \begin{equation}\begin{split}
&\sum_{i=1}^{n-2}   \int_{{\rm Cyl}(2\epsilon)}\bigg[m^{\alpha\beta}(x)\cdot \partial_{\alpha\beta}^2u+m^\alpha
(x)\cdot \partial_\alpha u  \bigg] \cdot X_i u dV_g
\\&\le \sum_{i=1}^{n-2} D\cdot \bigg[
\delta  \int_{{\rm Cyl}(2\epsilon)}|\nabla  [X_i u]|^2_g dV_g
+\delta^{-1} \int_{{\rm Cyl}(2\epsilon)}|\nabla   u|^2_g dV_g\bigg] 
\label{RHS bd}   \end{split}\end{equation}
   
   We prove \eqref{RHS bd}: For all terms in $m^\alpha(x)$ in the first factor the result is immediate using Cauchy-Schwarz; those terms are bounded by $\int_{{\rm Cyl}(2\epsilon)}|\nabla   u|^2_g dV_g$. 
   For the terms in $m^{\alpha\beta}(x)$, first consider the case were at least one of $\alpha,\beta$ is not 
   equal to $n-1,n$.  For those terms we again use Cauchy-Schwarz to bound them by the RHS of \eqref{RHS bd} 
   (for some $D$).  Finally,  for the terms where both indices $\alpha,\beta$ have one of the values $n-1,n$ we 
   perform an integration by parts in the direction $\alpha$ to write: 
   
   \begin{equation}\begin{split}
&   \int_{{\rm Cyl}(2\epsilon)}\bigg[m^{\alpha\beta}(x)\cdot \partial_{\alpha\beta}^2u  \bigg] \cdot X_i u dV_g=
-
    \int_{{\rm Cyl}(2\epsilon)}\bigg[m^{\alpha\beta}(x)\cdot \partial_{\beta}u  \bigg] \cdot \partial_\alpha[ X_i u] dV_g
  \\&  -\int_{{\rm Cyl}(2\epsilon)}\bigg[\partial_\alpha \bigg(\log\sqrt{\det(g)}\cdot m^{\alpha\beta}(x)\bigg)\cdot
   \partial_{\beta}u  \bigg] \cdot  X_i u dV_g
   \end{split}\end{equation}
   
   Using the regularity of the metric in our coordinate patch, we see that this is bounded by the RHS of \eqref{RHS bd}. 
   Increasing the constant $D$ if needed, we derive \eqref{RHS bd}.
   
   We can then choose $\delta$ small enough relative to $D$ and absorb the first term in the RHS of \eqref{RHS bd}
   into the first term in the LHS of \eqref{expanded} (making use of \eqref{IBP}); we derive:
   
   \[
\sum_{i=1}^{n-2}   \int_{{\rm Cyl}(2\epsilon)}|\nabla [X_iu]|_g^2 dV_g \le 2D\delta^{-1}
    \int_{{\rm Cyl}(2\epsilon)}|\nabla   u|^2_g dV_g.
   \]The RHS of the above is already bounded, since the solution $u$ to \eqref{the PDE} is known (or, for non-zero lower order coefficients in \eqref{operator}, assumed) 
    to lie in $H^1(M)$. 
   
    We can iterate this procedure to obtain the bounds: 
    
    \[
    \sum_{|I|\le {{d}}} \int_{{\rm Cyl}(2\epsilon)}|\nabla [X^I u]|_g^2 dV_g \le 2D\delta^{-{{d}}} \int_{{\rm Cyl}(2\epsilon)}|\nabla   u|^2_g dV_g.
    \]

Invoking the trace-theorem in each of the $n-2$ coordinates $x_1, \dots, x_{n-2}$ 
$(n-2)$ times (at each invocation we give up half a derivative in $L^2(M)$), 
we derive that for every point $Q\in  [-\epsilon,\epsilon]\times\dots\times[-\epsilon,\epsilon]$ and every multi-index $|I|\le 2$:
\[
\int_{D^2_P(\epsilon)} |\nabla [X^I u]|_g^2 dx_{n-1}dx_n<\infty. 
\]

This in particular implies that all  second  derivatives $\partial^2_{ij}u$ with $i,j\in \{1,\dots n-2\}$ lie
in $H^1(D^2(P))$.  We can invoke the Sobolev embedding theorem to derive that 
all second derivatives $\partial_{ij}u$ with $i,j\in \{1,\dots n-2\}$ lie in $L^p(D^2(P))$ for any $p<\infty$.

\end{proof}
   \medskip
   
   \newcommand{\sg}{\slashed{g}}
\newcommand{\beq}{\begin{equation}}
\newcommand{\eeq}{\end{equation}}

In particular, returning to the original equation 
${\cal L}[u]=\rho(x)$
and expanding it in the coordinates $x_1,\dots, x_n$ around any point 
$Q\in \partial E\subset \partial M$  we derive a Poisson-type  elliptic 
equation on \emph{each}     disc $D_Q^2(\e)$. To describe this equation let us denote by $\slashed{g}_Q$ the restriction $g|_{D^2_Q(\e)}$.

 We will expand the Laplacian with respect 
   to the coordinates $(x_1,\dots, x_{n-2}, x_{n-1}, x_n)$.  We use upper-case indices $A,B$ for entries between
    $1,\dots, n-2$. 
   We use Greek  indices $\alpha,\beta$  for $x_{n-1}, x_n$. 
Christoffel symbols will accordingly be denoted $\Gamma_{AB}^C, \Gamma_{AB}^\gamma, \Gamma_{A\beta}^C$ etc.~according 
to the values that the entries are allowed to take (upper case means values $1,\dots, n-2$ and lower-case means
 $n-1, n$).    
 
   We also let $\overline{\Delta}_\sg$ for the 
Laplacian  on each such leaf $D^2_Q(\e)$, with respect to the 
metric $\sg_Q$. 
   
    We next express the n-dimensional Laplacian $\Delta_g$  in these new coordinates: 
    The Einstein summation convention is 
    used for the upper-case and lower-case indices separately, so 
    ${}^{A\gamma}{}_{A\gamma}$ stands implicitly for 
    $\sum_{A=1}^{n-2}\sum_{\gamma=n-1}^n{}^{A\gamma}{}_{A\gamma}$. 
    
   \begin{equation}\begin{split}
    &  \Delta_g= g^{AB}\bigg[ \partial^2_{AB}
    -\Gamma_{AB}^C\partial_C-\Gamma_{AB}^\gamma\partial_\gamma\bigg] +g^{A\beta}
    \bigg[ \partial^2_{A\beta}
    -\Gamma_{A\beta}^C\partial_C-\Gamma_{A\beta}^c\partial_c\bigg]
    \\&+  \overline{\Delta}_{\sg}-g^{\alpha\beta}\Gamma_{\alpha\beta}^C\partial_C
   \end{split}\end{equation}

   This implies that a solution $u$ to ${\cal L}u=\rho(x)$ solves the following 
   PDE on each 2-dimensional leaf
   $(x_1,\dots, x_{n-2})={\rm fixed}$: 
   \begin{equation}\begin{split}
&-\overline{\Delta}_{\sg}u= g^{AB}\bigg[ \partial^2_{AB}u
    -\Gamma_{AB}^C\partial_Cu-\Gamma_{AB}^c\partial_cu\bigg] +g^{A\beta}
    \bigg[ \partial^2_{A\beta}u
    -\Gamma_{A\beta}^C\partial_Cu-\Gamma_{A\beta}^\gamma\partial_\gamma u\bigg]
    \\&  -g^{\alpha\beta}\Gamma_{\alpha\beta}^C\partial_Cu-a^I(x)\nabla_I u-a^i(x)\cdot \nabla_i u-V(x)\cdot u+\rho(x).
\label{Lapla in 2}
   \end{split}\end{equation}

We now restrict to the disc  $D^2(P)$. 
  (I.~e.~ we choose $Q=P$).  We recall  that on $D^2(P)$ $g^{A\alpha}=0$.

Using this vanishing of the cross terms and  in view of the regularity of the background metric $g$ and 
of the bounds we have obtained on first and second derivatives 
$\partial_{AB}u, \partial _{A\beta}u$ on $D^2_P(\e)$, we can express \eqref{Lapla in 2} 
in short as: 

\beq\label{elliptic simple}
-\overline{\Delta}_{\sg}u+ m^\alpha \partial_\alpha u=J_P,
\eeq
where all first order derivatives in the $x_{n-1}, x_n$ directions 
are moved to the LHS, and all other terms are kept on the right, and treated as a given RHS.  
As derived,  $m^\alpha$ are continuous functions and  $J_P(x_{n-1}, x_n )$ is an $L^p$ function over $D^2_P(\e)$.   Moreover 
the function $J_P(\cdot, \cdot)$ depends continuously on the point $P$, in the 
$L^p$ norm. 

Our next goal is to read off the asymptotics of $u$ at the point $Q$ from this 2-dimensional elliptic equation \eqref{elliptic simple}.  
   To do this, we perform a change of the two coordinates $x_{n-1}, x_n$:
    We define new coordinates $x,y$ which are isothermal 
  coordinates for  
  
  \[
  \sum_{\alpha,\beta=n-1, n}g_{\alpha\beta}(x_1, \dots x_{n-2})dx^\alpha dx^\beta. 
  \]
  This is done  by 
  solving $\overline{\Delta}_{\sg} y=0$ on $D^2_P(\e)$ with $y=x_{n-1}$ on 
  $\partial D^2_P(\e)$ and choosing $x$ to be the harmonic conjugate of 
  $y$, i.e. 
  \[
  (dx)_\alpha=\varepsilon_{\alpha\beta}(dy)^\beta,\] 
  normalized so that $x(P)=0$. 
  In particular in these coordinates 
 $\{y=0\}=\{x_n=0\}$ and $y>0$ over ${\rm int}\bigg[ D^2_P(\e)\bigg]$, by the maximum principle.  
Let us denote the coordinate transformation from $x_{n-1}, x_n$ to $(x,y)$ by $T$:
\beq
T(x_{n-1}, x_n)=(x,y). 
\eeq  
 
  In these coordinates $\{x,y\}$,  the metric $\sg_P$ acquires the form: 
  
  \[
  \sg_P(x,y)= \Omega^2(x,y,P)\cdot [dx^2+dy^2]. 
  \]
  Note the function $\Omega(x,y,P)$ is ${\cal C}^1$ in all its entries, and is uniformly bounded away from zero, 
  by the construction of the isothermal coordinates, using the fact that $dx\ne 0$ on $\partial D^2_P(\e)$. 
In particular 
\beq\label{Bounds}
M^{-1}\le |\nabla x|_{\sg}+|\nabla y|_{\sg}\le M
\eeq
for some universal $M>1$. 
  Let us denote by 
  $\Delta_{\mathbb{E}^2}$ the Laplacian $\partial_{xx}+\partial_{yy}$. 
  
\newcommand{\DE}{\Delta_{\mathbb{E}^2}}  
  
  We then transition to the new coordinate system $(x_1,\dots, x_{n-2}, x,y)$. 
For convenience,   we moreover restrict the new isothermal coordinates to the disc $\{x^2+y^2\le M^{-2}\e , y\ge 0\}$. 
We denote the latter by $D^2(\eta)\subset \mathbb{R}^2_+$, where $\eta= M^{-1}\sqrt{\e}$.  
 The equation \eqref{Lapla in 2} obtains the form: 
  
   \begin{equation}\begin{split}
&-\DE u  +\Omega^2 (x,y,P)[m_1 \partial_x u+m_2\partial_y u]  =\Omega^2(x,y,P)J_P(x,y). 
\label{Lapla in 2 again}
   \end{split}\end{equation}
  Here $m_1=m_1(x_1,\dots ,x_n), m_2=m_2(x_1,\dots, x_n)$ are continuous
   functions.

 Now, we consider  the same conformal mapping we did when $n=2$  \emph{to the $x,y$ coordinates 
  only}.  In other words, 
  we  construct new coordinates $(X,Y)$ 
  out of $x,y$ as follows:

   Letting $z=(x+iy)$, we define $(X+iY)= \zeta=\sqrt{z}$ as in the 
   previous section. Thus $\zeta$ takes values over the positive  quadrant $D_+(\sqrt{\eta})$ of radius 
   $\sqrt{\eta}$ in $\mathbb{R}^2$ centered at $(0,0)$. 
We note the transformation of the volume forms 

\beq\label{vol form}
dz\wedge d\overline{z}=4|\zeta|^2d\zeta \wedge d\overline{\zeta}.  
\eeq

    We define the new function $U(\zeta)= u(\zeta^2)$. 
   We invoke  the conformal 
    covariance of the Laplacian in two dimensions to derive:
    
    \beq\label{Lapl transfn}\begin{split}
   & [\partial_{XX}+\partial_{YY}] U=\partial_\zeta\partial_{\overline{\zeta}}U=|\zeta|^2\partial_z\partial_{\overline{z}}u
   \\& =|\zeta|^2\cdot [\partial_{xx}+\partial_{yy}]u= |\zeta|^2\cdot  \Omega^{2}J_P(x,y)-|\zeta| [{m}_X \partial_X u+{m}_Y\partial_Y u]
   \end{split} \eeq
   (The functions ${m}_X=m_1\cdot\Omega^2, {m}_Y=m_2\cdot \Omega^2$ are bounded functions).  The transformation law of the the volume form and of derivatives
   under this coordinate change  imply that
   \[
   \int_{D_+(\sqrt{\eta})} |\nabla_X U|^2 +|\nabla_Y U|^2 d\zeta d\overline{\zeta}<\infty.
   \]

   Our next goal is to use suitable reflections (in analogy with the 
   2-dimensional setting) to obtain an equation over the entire punctured 
   disc of radius $\sqrt{\eta}$ 
   in $\mathbb{R}^2$. This time, this involves reflecting both $U$ and also the first order coefficients, as well as  the function 
   $N=|\zeta|^2 \cdot  \Omega^{2}\cdot J_P(x,y)$, which we recall lies in $L^p(D^2(\sqrt{\eta}))$, $\forall p<\infty$. 

We note that $U$ satisfies that $ U=U_0={\rm Const}$ on the positive $X$-axis and $\partial_XU=0$ on the positive $Y$-axis. 
In fact we replace $U$ by $U-U_0$ by slight abuse of notation. 

 We then construct an extended $\tilde{U}(X,Y)$ by first performing 
 an even reflection across the $Y$-axis. Note that the resulting function is ${\cal C}^{2}$, 
 except at the origin.  We also extend $N$ by  an 
 even reflection across the $Y$-axis,   which yields an $L^p$ function.    ${m}_X$ is extended by 
 an odd reflection and ${m}_Y$ by an even reflection. 
  The resulting functions are denoted by $\tilde{m}_X, \tilde{m}_Y$; they are continuous functions. 
 
 The resulting functions 
 $\tilde{U}, \tilde{N}, \tilde{m}_X, \tilde{m}_Y$ are defined over the half-disc 
 $D(\sqrt{\eta})\bigcap \{Y> 0\}$. It  follows that $\tilde{U}, \tilde{N}$ satisfy: 
 
 \[
 [\partial_{XX}+\partial_{YY}] \tilde{U} +|\zeta| [{m}_X \partial_X \tilde{U}+{m}_Y\partial_Y \tilde{U}]=\tilde{N}
 \]
 on that domain, in the strong sense. We then consider the new function $\tilde{U}$, which 
 now vanishes on $\{Y=0\}\bigcap D^2(\sqrt{\eta})$. We perform an odd 
 reflection of this function across the $X$-axis , and also perform an odd reflection of 
 $\tilde{N}$ across the $X$-axis,  which again yields a ${\cal C}^{1,1}$  function away from the origin. 
 We also perform  an odd reflection of $\tilde{m}_X$ and an even reflection for $\tilde{m}_Y$. 
  Denote the functions we obtain on all of 
 $D^2(\sqrt{\eta})\setminus \{(0,0)\}$ by $R$ and $\overline{K}$, $\overline{m}_X,\overline{m}_Y$.

 These functions then satisfy 
  
   \beq\label{new new elliptic}
 \DE (R)  +|\zeta| [\overline{m}_X \partial_X R+\overline{m}_Y\partial_Y R]=\overline{K},
 \eeq 
on  $D^2(\sqrt{\eta})\setminus \{(0,0)\}$,  again in the strong sense.   We note further that 
 $\nabla R\in L^2( dXdY)$) and 
  $\overline{K}\in L^p(DXDY)$, $\forall p<\infty$.  We recall that $R$ is a ${\cal C}^{1,1}$  function on the boundary $C(\sqrt{\eta})=\partial D(\sqrt{\eta})$. 
  (The function is in fact ${\cal C}^2$ except precisely on the $X$-axis)
 We can then construct a ${\cal C}^{1,1}$  extension $R^\sharp(X,Y)$ 
of $R(X,Y)|_{C(\sqrt{\eta})}$ which vanishes in an open neighborhood of the origin and moreover has the same (even or odd) 
reflection symmetries across the axes as $R(X,Y)$.  
Considering the function $\delta R(X,Y)=R(X,Y)-R^\sharp(X,Y)$ we then obtain a function with zero Dirichlet boundary data 
on $C(\sqrt{\eta})$, which solves an equation: 

    \beq\label{new new elliptic delta}
 \DE (\delta R)  +|\zeta| [\overline{m}_X \partial_X [\delta R]+\overline{m}_Y\partial_Y [\delta R]]=\overline{E}.
 \eeq 
  $\overline{E}\in L^p(dXdY)$, $\forall p<\infty$.   $\delta R$ vanishes on the boundary $C(\sqrt{\eta})$
  and so far we have obtained that it  lies in $W^{1,2}(dXdY)$.  We will show that $\delta R(X,Y)$ lies in $C^{1,\alpha}\left(D^2(\sqrt{\eta})\right)$ 
  for any $\alpha<1$.

To derive this,  consider the Dirichlet Green's function $G(X,Y)$ on the 
 disc $D^2(\sqrt{\eta})$,  and we  convolve \eqref{new new elliptic delta} against that function. 
 We note that since we have ${\cal C}^2$ regularity of 
 $\delta R(X,Y)$   
 near all points except for the origin  the convolution 
\[
B(X,Y)=G*\overline{E}
\]
is well-defined at all points except for the origin. 
Since $\overline{E}\in L^p(D^2(\sqrt{\eta}))$ $B$ lies in  $W^{2,p}(D^2(\sqrt{\eta}))$,  by the mapping properties of the Green's function of the disc.  
Sobolev embedding then implies that $B(X,Y)$  lies in $C^{1,\alpha}(D^2(\sqrt{\eta}))$ 
for any $\alpha<1$.  We also recall that since  $\delta R$ is 
smooth away from the origin, 
we derive that at all points away from the origin: 
\[
\bigg\{ [\DE \delta R]*G\bigg\}(X,Y)=\delta R (X,Y). 
\]

 We also  derive:
\beq
\sum_{|J|\le 2}\|D^J \bigg[  [|\zeta|\overline{m}_X \partial_X [\delta R]+|\zeta|\overline{m}_Y\partial_Y [\delta R]]*G  \bigg]\|_{L^p\big(D^2(\sqrt{\eta})\big)}\le  M\cdot \|\delta R \|_{W^{1,p}\left( D^2(\sqrt{\eta})\right)}
\eeq
The constant $M$ is uniform, and independent of $\eta$ (it depends 
only on the ${\cal C}^2$ norm of $\overline{m}_X, \overline{m}_Y$). 
We can then invoke the Poincare inequality (recall $\delta R$ vanishes on the positive $X$-axis) to derive: 

\beq\begin{split}
&\sum_{|J|\le 2}\|D^J \bigg[ \sqrt{X^2+Y^2}\cdot  [\overline{m}_X \partial_X [\delta R]+\overline{m}_Y\partial_Y [\delta R]]*G  \bigg]\|_{L^p(D^2(\sqrt{\eta}))}
\\&\le  M\cdot\sqrt{\eta} \|\delta R \|_{W^{2,p}\left( D^2(\sqrt{\eta})\right)}.
\end{split}\eeq

Thus, convolving \eqref{new new elliptic delta} with $G$ and taking up to two derivatives, 
taking the $L^p$ norm of the resulting equations and summing, we derive: 

\[
(1-M\sqrt{\eta}) \|\delta R  \|_{W^{2,p}(D^2(\sqrt{\eta}))}\le M\| \overline{E}\|_{L^p\big(D^2(\sqrt{\eta})\big)}<+\infty
\]
Choosing $\eta$ small enough we derive the finiteness of $ \|\delta R  \|_{W^{2,p}\big(D(\sqrt{\eta})\big)}$.
By Sobolev embedding, we derive the finiteness of $ \|\delta R  \|_{{\cal C}^{1,\alpha}\big(D(\sqrt{\eta})\big)}$, 
$\alpha<1$.

This then implies that the vector-valued function $(\partial_X \delta R, \partial_Y \delta R)$ lies in 
${\cal C}^\alpha\left( D(\sqrt{\eta})\right)$. 
 Consider $\tilde{\delta R}(x,y)= \delta R(\sqrt{z})$ (where $z=x+i\cdot y$).  Using the transformation law of the coordinates $(X,Y)$ to $(x,y)$ 
 as well as the vanishing of $\partial_X \tilde{\delta R}(X, Y=0)$ 
and  $\partial_Y \tilde{\delta R}(X=0, Y)$ 
we derive: 
\beq\label{lead and rest}
\bigg(\partial_x \tilde{\delta R}(x,y), \partial_y \tilde{\delta R}(x,y)\bigg)=\tilde{A}\cdot (\sqrt{x^2+y^2})^{-1/2}\cdot (\sin(\frac{\phi}{2}), \cos(\frac{\phi}{2}))+O((\sqrt{x^2+y^2})^{-1/2+\alpha}). 
\eeq
 
 Now, again 
  identifying $(x,y)$ 
 with $x+i\cdot y$ 
we then transform back to coordinates 
$(x_{n-1}, x_n)= T^{-1}(x+i\cdot y)$. 
Defining $b(0,\dots, 0, x_{n-1}, x_n)$  to be the second (remainder term) in \eqref{lead and rest} after this change of coordinates 
  we derive 
  \eqref{stronger claim},  for $(x_1,\dots, x_{n-2})=(0,\dots, 0)$,  where
 $b(0,\dots ,0,x_{n-1},x_n)$ is a continuous function in $x_{n-1}, x_n$.  Now,  we make note that the
 estimate we have obtained applies to \emph{every} set of points 
 $(x_1,\dots , x_{n-2})$ (choosing 
$(x_1,\dots, x_{n-2})=(0,\dots ,0)$ was just  a coordinate normalization).  This thus implies \eqref{stronger claim}, 
except we have no 
regularity for the functions $A(x_1,\dots, x_{n-2}, x_{n-1}, x_n)$,  $b(x_1,\dots, x_{n-2}, x_{n-1}, x_n)$ in the 
coordinates $x_1,\dots, x_{n-2}$. 
To derive the continuity of the two functions in these coordinates, we just return to the PDEs  
\eqref{Diffd PDE},  \eqref{Lapla in 2}, \eqref{Lapla in 2 again}, \eqref{new new elliptic}
 that were used in 
our derivation,  and note that the RHSs, the coefficients and the boundary values of the unknowns are 
${\cal C}^1$ in the coordinates $x_1,\dots, x_{n-2}$ (measured in their respective norms over $x_{n-1}, x_n)$). 
 This yields the continuity of the functions 
\[
A(x_1,\dots,x_{n-2}), b(x_1,\dots, x_n)
\] in 
$x_1,\dots, x_{n-2}$ also.

   \end{proof}
   \subsubsection*{Extension to domains with corners.}

   We conclude the paper by noting an extension of the main theorem which captures the more realistic model where
   the electrodes are smooth but stiff,  and thus the surfaces $E_i$ form a \emph{corner} with the boundary 
   $\partial \Omega$ at the edges $\partial E_i$.  To not burden the reader with excessive notation,  we present this 
   extension only in dimension $2$--the analogue of Theorem \ref{gen thm}
   would be too long to sketch.  
  In two dimensions, the setting would be of the electrode $e_i$'s being a smooth curves in $\mathbb{R}^2$,  and also 
  $\partial\Omega\setminus \bigcup_{i=1}^Ne_i$ consists of a disjoint union of closed segments.  The assumption is 
  that at each lf the two endpoints $L_i, R_i$ of $e_i$ the two smooth curves $e_i$
   and $\partial\Omega\setminus e_i$  form an angle 
  $\phi_{i,L}, \phi_{i,R}\in (0,2\pi)$. (If the angle where $\pi$, the boundary $\partial\Omega$ would be  ${\cal C}^1$
   near those points). 
   
  The generalization of our  theorem 
\ref{Them 2D} in this setting of corners  (denoting $\phi_{i,L}$, $\phi_{i,R}$ by 
  $\varphi$ for short, and also the distance ${\rm dist}(\cdot, L_i), {\rm dist}(\cdot, R_i)$ from either $L_i, R_i$ by
  ${\rm dist}$),  is that on the electrode $e_i$ upon approach 
  to either of its endpoints, the Neumann data has the expansion:
  
  \begin{equation}\label{corners claim}
 | \partial_y u|=A\cdot {\rm dist}^{\frac{\pi}{2\varphi}-1}+ 
 O({\rm dist}^{\frac{\pi}{2\varphi}-1+\alpha\cdot \frac{\pi}{\varphi}})). 
  \end{equation}
  
    Let us sketch how the proof of this reduces readily to the setting of domains $\Omega$ with ${\cal C}^1$ boundary:
  
  In dimension 2 locally near each endpoint of any electrode we 
 define a new complex coordinate $\tilde{z}$ from the 
  initial coordinate $z=x+i\cdot y$ via:  
  $\tilde{z}=z^{\pi/\varphi}$. 
  The resulting boundary then becomes ${\cal C}^1$ (in fact 
  ${\cal C}^{1,1}$), and the static potential $u$ 
 still satisfies $\Delta u=0$ (with respect to the new coordinates), and moreover $u$ remains constant on 
  the electrode $e_i$ and $\partial_y u=0$ off of the electrode. 
  We can thus 
  invoke theorem \ref{Them 2D} in the new coordinate $\tilde{z}$ and in 
  the end transform back to the original coordinate $z$.  
 Keeping track of the coordinate and vector field transformations, we derive \eqref{corners claim}.

   \subsubsection*{Acknowledgements}
   
   I am grateful to Adam Stinchcombe for raising the question addressed in this paper. I am also indebted to Petri Ola 
   for 
guiding me through literature that is pertinent to  this topic. This research was partly supported by NSERC grant RGPIN-2019-06946.

   \bibliographystyle{plain}
\bibliography{references.bib}

\begin{thebibliography}{10}

\bibitem{AlChMa}
Yurij~A. Alkhutov, Gregory~A. Chechkin, and Vladimir~G. Maz'ya.
\newblock Boyarsky-{M}eyers estimate for solutions to {Z}aremba problem.
\newblock {\em Arch. Ration. Mech. Anal.}, 245(2):1197--1211, 2022.

\bibitem{Darbasetal}
Marion Darbas, Jérémy Heleine, Renier Mendoza, and Arrianne~Crystal Velasco.
\newblock Sensitivity analysis of the complete electrode model for electrical
  impedance tomography.
\newblock {\em AIMS Mathematics}, 6(7):7333--7366, 2021.

\bibitem{DelKress}
Fabrice Delbary and Rainer Kress.
\newblock Electrical impedance tomography using a point electrode inverse
  scheme for complete electrode data.
\newblock {\em Inverse Problems and Imaging}, 5(2):355--369, 2011.

\bibitem{Eskin}
G.~I. Eskin.
\newblock {\em Boundary value problems for elliptic pseudodifferential
  equations}, volume~52 of {\em Translations of Mathematical Monographs}.
\newblock American Mathematical Society, Providence, RI, 1981.
\newblock Translated from the Russian by S. Smith.

\bibitem{Grisvard}
Pierre Grisvard.
\newblock {\em Elliptic problems in nonsmooth domains}, volume~69 of {\em
  Classics in Applied Mathematics}.
\newblock Society for Industrial and Applied Mathematics (SIAM), Philadelphia,
  PA, 2011.
\newblock Reprint of the 1985 original [MR0775683], With a foreword by Susanne
  C. Brenner.

\bibitem{HAKULA}
Harri Hakula, Nuutti Hyvönen, and Tomi Tuominen.
\newblock On the hp-adaptive solution of complete electrode model forward
  problems of electrical impedance tomography.
\newblock {\em Journal of Computational and Applied Mathematics},
  236(18):4645--4659, 2012.
\newblock FEMTEC 2011: 3rd International Conference on Computational Methods in
  Engineering and Science, May 9–13, 2011.

\bibitem{Hyvonen}
Nuutti Hyv\"{o}nen.
\newblock Complete electrode model of electrical impedance tomography:
  Approximation properties and characterization of inclusions.
\newblock {\em SIAM Journal on Applied Mathematics}, 64(3):902--931, 2004.

\bibitem{Kondratev}
V.~A. Kondratev.
\newblock Boundary value problems for elliptic equations in domains with
  conical or angular points.
\newblock {\em Trudy Moskov. Mat. Ob\v{s}\v{c}.}, 16:209--292, 1967.

\bibitem{KozlovMazyaRossmann}
V.~A. Kozlov, V.~G. Maz'ya, and J.~Rossmann.
\newblock {\em Elliptic boundary value problems in domains with point
  singularities}, volume~52 of {\em Mathematical Surveys and Monographs}.
\newblock American Mathematical Society, Providence, RI, 1997.

\bibitem{Pommerenke}
Ch. Pommerenke.
\newblock {\em Boundary behaviour of conformal maps}, volume 299 of {\em
  Grundlehren der mathematischen Wissenschaften [Fundamental Principles of
  Mathematical Sciences]}.
\newblock Springer-Verlag, Berlin, 1992.

\bibitem{RempSchulzePaper}
Stephan Rempel and Bert-Wolfgang Schulze.
\newblock Parametrices and boundary symbolic calculus for elliptic boundary
  problems without the transmission property.
\newblock {\em Math. Nachr.}, 105:45--149, 1982.

\bibitem{SomChenIsaac}
Erkki Somersalo, Margaret Cheney, and David Isaacson.
\newblock Existence and uniqueness for electrode models for electric current
  computed tomography.
\newblock {\em SIAM J. Appl. Math.}, 52(4):1023--1040, 1992.

\bibitem{Vauhkonenetal}
P.J. Vauhkonen, M.~Vauhkonen, T.~Savolainen, and J.P. Kaipio.
\newblock Three-dimensional electrical impedance tomography based on the
  complete electrode model.
\newblock {\em IEEE Transactions on Biomedical Engineering}, 46(9):1150--1160,
  1999.

\bibitem{Zaremba}
Stanislaw Zaremba.
\newblock Sur un proble\'me mixte relatif a\' l' \'equation de laplace.
\newblock {\em Bulletin international de l'Académie des Sciences de Cracovie,
  Classe des sciences mathématiques et naturelles}, A:313--–344, 1910.

\end{thebibliography}

        \end{document}